\documentclass[11pt]{amsart}
\usepackage[all]{xy}
\usepackage{bbm}

\newtheorem{theorem}{Theorem}[section]
\newtheorem{prop}[theorem]{Proposition}
\newtheorem{lemma}[theorem]{Lemma}
\newtheorem{corol}[theorem]{Corollary}

\theoremstyle{definition}
\newtheorem{defin}[theorem]{Definition}

\theoremstyle{remark}
\newtheorem{remk}[theorem]{Remark}

\numberwithin{equation}{section}

\title[Representations of nodal algebras]{Representations of nodal algebras of type $\rA$}
\author{Yuriy A. Drozd}
\author{Vasyl V. Zembyk}
\address{Institute of Mathematics, National Academy of Sciences of Ukraine,
01601 Kyiv, Ukraine}
\email{y.a.drozd@gmail.com, drozd@imath.kiev.ua}
\urladdr{www.imath.kiev.ua/$\sim$drozd}
\email{vaszem@rambler.ru}

\def\rA{\mathrm A}	\def\nR{\mathrm r}
\def\bA{\mathbf A}	\def\bH{\mathbf H}
\def\bB{\mathbf B}	\def\bS{\mathbf S}
\def\bF{\mathbf F}	\def\bG{\mathbf G}
\def\bC{\mathbf C}
\def\aK{\mathbbm k}	\def\fG{\mathbf g}
\def\fF{\mathbf f}	\def\bE{\mathbf E}
\def\dS{\mathfrak S}	\def\dC{\mathfrak C}
\def\dL{\mathfrak L}	\def\dM{\mathfrak M}
\def\kC{\mathcal C}

\def\Ga{\Gamma}
\def\eps{\varepsilon}
\def\al{\alpha}		\def\be{\beta}
\def\la{\lambda}	\def\ga{\gamma}

\def\oA{\bar\bA}
\def\oB{\bar\bB}
\def\oH{\bar\bH}
\def\tS{\widetilde\bS}
\def\tH{\widetilde\bH}
\def\tA{\widetilde\rA}
\def\tB{\widetilde B}

\def\oM{\bar M}
\def\tF{\tilde f}

\def\lb{\textup{(}} 	\def\rb{\textup{)}}

\def\sbe{\subseteq}	\def\sb{\subset}
\def\spe{\supseteq}	
\def\bop{\bigoplus}	\def\+{\oplus}
\def\*{\otimes}		\def\xx{\times}
\def\={\setminus}	\def\ito{\stackrel\sim\to}
\def\bap{\bigcap}
\def\ol{\overline}	\def\ti{\tilde}
\def\str{\stackrel}

\def\md{\mbox{-}\mathrm{mod}}
\def\rmd{\mathrm{mod}\mbox{-}}
\def\ccdot{\boldsymbol{\cdot}}
\def\id{\mathrm{Id}}
\def\im{\mathop\mathrm{Im}\nolimits}
\def\hom{\mathop\mathrm{Hom}\nolimits}
\def\Ker{\mathop\mathrm{Ker}\nolimits}
\def\End{\mathop\mathrm{End}\nolimits}
\def\len{\mathop\mathrm{length}\nolimits}
\def\Ar{\mathop\mathrm{Ar}}

\def\rad{\mathop\mathrm{rad}\nolimits}
\def\Mat{\mathop\mathrm{Mat}\nolimits}

\def\row#1#2{\left( #1_1 , #1_2 , \dots , #1_{#2} \right)}
\def\set#1{\left\{\,#1\,\right\}}
\def\setsuch#1#2{\left\{\,#1\mid #2\,\right\}}
\def\lst#1#2{ #1_1 , #1_2 , \dots , #1_{#2} }
\def\lsto#1#2{ #1_0 , #1_1 , \dots , #1_{#2} }
\def\gnr#1{\langle\,#1\,\rangle}

\def\mtr#1{\begin{pmatrix}#1\end{pmatrix}}

\def\fdd{finite dimensional}
\def\iff{if and only if }
\def\oc{one-to-one correspondence}

\begin{document}
\maketitle

\begin{abstract}
 We define nodal \fdd\ algebras and describe their structure over an algebraically closed
 field. For a special class of such algebras (\emph{type} $\rA$) we find a criterion of
 tameness.
\end{abstract}

\tableofcontents

\section*{Introduction}

 Nodal (in\fdd) algebras first appeared (without this name) in the paper \cite{dr} as
 \emph{pure noetherian}%
\footnote{\, Recall that \emph{pure noetherian} means noetherian without minimal submodules.}
 algebras that are tame with respect to the classification of finite length modules. In \cite{bd1}
 their derived categories of modules were described showing that such algebras are also derived tame.
 Voloshyn \cite{vol} described their structure.
 The definition of nodal algebras can easily be applied to \fdd\ algebras too. The simplest examples
 show that in \fdd\ case the above mentioned results are no more true: most nodal algebras are wild. 
 It is not so strange, since they are obtained from hereditary algebras, most of which are also wild,
 in contrast to pure noetherian case, where the only hereditary algebras are those of type $\tA$.
 Moreover, even if we start from hereditary algebras of type $\rA$, we often obtain wild nodal 
 algebras. So the natural question arise, which nodal algebras are indeed tame, at least if we
 start from a hereditary algebra of type $\rA$ or $\tA$. In this paper we give an answer to this 
 question (Theorem~\ref{main}). 
 
 The paper is organized as follows. In Section 1 we give the definition of nodal algebras and
 their description when the base field is algebraically closed. This description is alike that
 of \cite{vol}. Namely, a nodal algebra is obtained from a hereditary one by two operations
 called \emph{gluing} and \emph{blowing up}. Equivalently, it can be given by a quiver and a
 certain symmetric relation on its vertices. In Section 2 we consider a special sort of gluings
 which do not imply representation types. In Section 3 \emph{gentle} and \emph{skewed-gentle} 
 nodal algebras are described. Section 4 is devoted to a class of nodal algebras called 
 \emph{exceptional}. We determine their representation types. At last, in Section 5 we summarize
 the obtained results and determine representation types of all nodal algebras of type $\rA$.
 
\section{Nodal algebras}
\label{sec1} 

 We fix an algebraically closed field $\aK$ and consider algebras over $\aK$. 
 Moreover, if converse is not mentioned, all algebras are supposed to be finite dimensional.
 For such an algebra $\bA$ we denote by $\bA\md$ the category of (left) finitely generated
 $\bA$-modules. If an algebra $\bA$ is \emph{basic} (i.e. $\bA/\rad\bA
 \simeq\aK^s$), it can be given by a \emph{quiver} (oriented graph) and relations (see
 \cite{ars} or \cite{dk}). Namely, for a quiver $\Ga$ we denote by $\aK\Ga$ the 
 \emph{algebra of paths} of $\Ga$ and $J_\Ga$ be its ideal generated by all arrows. 
 Then every basic algebra is isomorphic to $\aK\Ga/I$, where $\Ga$ is a quiver and $I$ is an 
 ideal of $\aK\Ga$ such that $J_\Ga^2\spe I\spe J_\Ga^k$ for some $k$. Moreover, the
 quiver $\Ga$ is uniquely defined; it is called the \emph{quiver of the algebra} $\bA$.
 For a vertex $i$ of this quiver we denote by $\Ar(i)$ the set of arrows incident to $i$.
 Under this presentation $\rad\bA=J_\Ga/I$, so $\bA/\rad\bA$ can be identified with the
 vector space generated by the ``empty paths'' $\eps_i$, where $i$ runs the vertices of $\Ga$. 
 Note that $1=\sum_i\eps_i$ is a decomposition of the unit of $\bA$ into a sum of primitive 
 orthogonal idempotents. Hence simple $\bA$-modules, as well as 
 indecomposable projective $\bA$-modules are in \oc\ with the vertices of the quiver $\Ga$. 
 We denote by $\oA_i$ the simple module corresponding to the vertex $i$ and by $\bA_i=\eps_i\bA$ 
 the \emph{right} projective $\bA$-module corresponding to this vertex. We also write $i=\al^+$
 ($i=\al^-$) if the arrow $\al$ ends (respectively starts) at the vertex $i$.
 Usually the ideal $I$ is given by a set of generators $R$ which is then called the
 \emph{relations of the algebra} $\bA$. Certainly, the set of relations (even a minimal one) 
 is far from being unique. 
  An arbitrary algebra can be given by a quiver $\Ga$ with relations and \emph{multiplicities} 
  $m_i$ of the vertices $i\in\Ga$. 
 Namely, it is isomorphic to $\End_\bA P$, where $\bA$ is the basic algebra of $\bA$
 and $P=\bop_im_i\bA_i$. (We denote by $mM$ the direct sum of $m$ copies of module $M$.)
  Recall also that path algebras of quivers without (oriented) cycles
 are just all \emph{hereditary} basic algebras (up to isomorphism) \cite{ars,dk}. 
 
\begin{defin}\label{nodal} 
 A (\fdd) algebra $\bA$ is said to be \emph{nodal} if there is a hereditary algebra $\bH$
 such that 
 \begin{enumerate}
 \item $\rad\bH\sb\bA\sbe\bH$,
 \item $\rad\bA=\rad\bH$,
 \item $\len_\bA(\bH\*_\bA U)\le2$ for each simple $\bA$-module $U$.
 \end{enumerate}
 We say that the nodal algebra $\bA$ \emph{is related to the hereditary algebra} $\bH$.
\end{defin}

\begin{remk}\label{remark} 
 From the description of nodal algebras it follows
 that the condition (3) may be replaced by the opposite one:
 \begin{enumerate}
 \item[($3'$)]  $\len_\bA(U\*_\bA\bH)\le2$ for each simple right $\bA$-module $U$
 \end{enumerate} 
 (see Corollary~\ref{opposite} below).
\end{remk}

\begin{prop}\label{morita} 
 If an algebra $\bA'$ is Morita equivalent to a nodal algebra $\bA$ related to a hereditary
 algebra $\bH$, then $\bA'$ is a nodal algebra related to a hereditary algebra $\bH'$ that is
 Morita equivalent to $\bH$.
\end{prop}
\begin{proof}
 Denote $J=\rad\bH=\rad\bA$. 
 Let $P$ be a projective generator of the category $\rmd\bA$ of right $\bA$-modules
 such that $\bA'\simeq\End_\bA P$. Then also 
 $\bA\simeq \End_{\bA'}P\simeq P^\vee\*_{\bA'}P$, where 
 $P^\vee=\hom_{\bA'}(P,\bA')\simeq \hom_\bA(P,\bA)$.
 Let $P'=P\*_\bA\bH$. Then $P'$ is a projective generator of the  
 category $\rmd\bH$. Set $\bH'=\End_\bH P'$. Note that $\hom_\bH(P',M)\simeq\hom_\bA(P,M)$ 
 for every right $\bH$-module $M$. In particular, $\End_\bH P'\simeq\hom_\bA(P,P')$. 
 Hence, the natural map $\bA'\to\bH'$ is a monomorphism. Moreover, since $P'J=PJ$,
\begin{multline*}
\rad\End_\bH P'=\hom_\bH(P',P'J)\simeq\\
 \simeq \hom_\bA(P,P'J)=\hom_\bA(P,PJ)=\rad\End_\bA P
\end{multline*}
(see \cite[Chapter III, Exercise 6]{dk}). Thus $\rad\bA'=\rad\bH'\sb\bA'\sbe\bH'$.
 Every simple $\bA'$-module is isomorphic to $U'=P\*_\bA U$ for some simple $\bA$-module
 $U$. Therefore
 \begin{align*}
  \bH'\*_{\bA'}U'&=\bH'\*_{\bA'}(P\*_\bA U)\simeq (\bH'\*_{\bA'}P)\*_\bA U \simeq\\
   &\simeq (P\*_\bA\bH)\*_\bA U\simeq P\*_\bA(\bH\*_\bA U),\\
  \intertext{since}
  \bH'\*_{\bA'}P&\simeq \hom_\bA(P,P\*_\bA\bH)\*_{\bA'}P \simeq
    ((P\*_\bA\bH)\*_\bA P^\vee)\*_{\bA'}P \simeq\\
   &\simeq (P\*_\bA\bH)\*_\bA (P^\vee\*_{\bA'}P) \simeq P\*_\bA(\bH\*_\bA\bA)
   \simeq P\*_\bA\bH.
 \end{align*}
 Hence $\len_{\bA'}(\bH'\*_{\bA'}U')=\len_\bA(\bH\*_\bA U)\le2$, so $\bA'$ is nodal.
 \end{proof}
 
 This proposition allows to consider only \emph{basic} nodal algebras $\bA$, i.e. such
 that $\bA/\rad\bA\simeq\aK^m$ for some $m$. We are going to present a construction that gives all
 basic nodal algebras. 
  
 \begin{defin}\label{gluing} 
 Let $\bB$ be a basic algebra, $\oB=\bB/\rad\bB=\bop_{i=1}^m\oB_i$, where $\oB_i\simeq\aK$
 are simple $\bB$-modules. 
 \begin{enumerate}
 \item Fix two indices $i,j$.
  Let $\oA$ be the subalgebra of $\oB$ consisting of all $m$-tuples $\row\la m$
 such that $\la_i=\la_j$, $\bA$ be the preimage of $\oA$ in $\bB$. We say that $\bA$ 
 \emph{is obtained from $\bB$ by gluing the components $\oB_i$ and $\oB_j$} (or the
 corresponding vertices of the quiver of $\bB$).
 
 \item Fix an index $i$. Let $P=2\bB_i\+\bop_{k\ne i}\bB_k$, $\bB'=\End_\bB P$, 
 $\oB'=\oB/\rad\oB=\prod_{k=1}^m\oB'_i$, where $\oB'_i\simeq\Mat(2,\aK)$ and 
 $\oB'_k\simeq\aK$ for $k\ne i$. Let $\oA'$ be the subalgebra of $\oB'$ consisting of
 all $m$-tuples $\row bm$ such that $b_i$ is a diagonal matrix, and $\bA$ be
 the preimage of $\oA$ in $\bB'$. We say that $\oA$ \emph{is obtained from $\bB$ by
 blowing up the component $\oB_i$} (or the corresponding vertex of the quiver $\bB$).
  \end{enumerate}
 \end{defin}
 
 This definition immediately implies the following properties.
 
 \begin{prop}\label{properties} 
  We keep the notations of Definition~\ref{gluing}.
  \begin{enumerate}
  \item If $\bA$ is obtained from $\bB$ by gluing components $\oB_i$ and $\oB_j$,
  then it is basic and $\bA/\rad\bA=\oA_{ij}\xx\prod_{k\notin\{i,j\}}\oB_k$, where
  $\oA_{ij}=\setsuch{(\la,\la)}{\la\in\aK}\sb\oB_i\xx\oB_j$. Moreover, $\rad\bA=\rad\bB$
  and $\bB'\*_\bA\oA_{ij}\simeq\oB_i\xx\oB_j$.
  
  \item If $\bA$ is obtained from $\bB$ by blowing up a component $\oB_i$,
  then it is basic and $\bA/\rad\bA=\oA_{i1}\xx\oA_{i2}\xx\prod_{k\ne i}\oB_k$, where
  $\oA_{is}=\setsuch{\la e_{ss}}{\la\in\aK}$ and $e_{ss}\ (s\in\{1,2\})$ denote the
  diagonal matrix units in $\oB_i\simeq\Mat(2,\aK)$. Moreover, $\rad\bA=\rad\bB'$ and
  $\bB\*_\bA\oA_{is}\simeq V$, where $V$ is the simple $\oB'_i$-module.
  \end{enumerate}
  
 \emph{We call the component $\oA_{ij}$ in the former case and the components $\oA_{is}$ in the 
  latter case \emph{the new components} of $\bA$. We also identify all other simple components
  of $\oA$ with those of $\oB$.} 
 \end{prop}

  \begin{prop}\label{relations} 
  Under the notations of Definition~\ref{gluing} suppose that the algebra $\bB$ is given
  by a quiver $\Ga$ with a set of relations $R$. 
  \begin{enumerate}
  \item Let $\bA$ be obtained from $\bB$ by gluing the components corresponding
  to vertices $i$ and $j$. Then the quiver of $\bA$ is obtained from $\Ga$ by identifying
  the vertices $i$ and $j$, while the set of relations for $\bA$ is $R\cup R'$, where
  $R'$ is the set of all products $\al\be$, where $\al$ starts at $i$ \lb or at $j$\rb\ 
  and $\be$ ends at $j$ \lb respectively, at $i$\rb.
  
  \item  Let $\bA$ be obtained from $\bB$ by blowing up the component corresponding to
  a vertex $i$ and there are no loops at this vertex.\!%
  \footnote{\,One can modify the proposed procedure to include such loops, but this modification
  looks rather cumbersome and we do not need it.}
    Then the quiver of $\bA$ and the 
  set of relations for $\bA$ are obtained as follows:
   \begin{itemize}
   \item replace the vertex $i$ by two vertices $i'$ and $i''$;
   \item replace every arrow $\al:j\to i$ by two arrows $\al':j\to i'$ and $\al'':j\to i''$;
   \item replace every arrow $\be:i\to j$ by two arrows $\be':i'\to j$ and $\be'':i''\to j$;   
   \item replace every relation $\nR$ containing arrows from $\Ar(i)$ by two relations
   $\nR'$ and $\nR''$, where $\nR'$ \lb$\nR''$\rb\ is obtained from $\nR$ by replacing each arrow 
   $\al\in \Ar(i)$ by $\al'$ \lb respectively, by $\al''$\rb;
   \item  keep all other relations;
   \item for every pair of arrows $\al$ starting at $i$ and $\be$ ending at $i$ add a relation
   $\al'\be'=\al''\be''$.
   \end{itemize}
  \end{enumerate}
  \end{prop}
  
 \begin{defin}\label{suitable} 
   We keep the notations of Definition~\ref{gluing} and choose pairwise different 
   indices $\lst i{r+s}$ and $\lst jr$ from $\set{1,2,\dots,m}$. We construct the 
   algebras $\lsto\bA{r+s}$ recursively:
   \begin{itemize}
   \item[] $\bA_0=\bB$;
   \item[] for $1\le k\le r$ the algebra $\bA_k$ is obtained from $\bA_{k-1}$ by gluing
   the components $\oB_{i_k}$ and $\oB_{j_k}$;
   \item[] for $r<k\le r+s$ the algebra $\bA_k$ is obtained from $\bA_{k-1}$ by blowing
   up the component $\oB_{i_k}$.
   \end{itemize}
   In this case we say that the algebra $\bA=\bA_{r+s}$ is obtained from $\bB$ by the 
   \emph{suitable sequence of gluings and blowings up} defined by the sequence of indices
   $(\lst i{r+s},\lst jr)$. Note that the order of these gluings and blowings up does not 
   imply the resulting algebra $\bA$.
   
   Usually such sequence of gluings and blowings up is given by a symmetric relation $\sim$
   (not an equivalence!) on the vertices of the quiver of $\bB$ or, the same, on the set of simple  
   $\bB$-modules $\oB_i$: we set $i_k\sim j_k$ for $1\le k\le r$ and $i_k\sim i_k$ for $r<k\le r+s$.
   Note that $\#\setsuch{j}{i\sim j}\le1$ for each vertex $i$. 
 \end{defin}
 
 \begin{theorem}\label{description} 
  A basic  algebra $\bA$ is nodal \iff it is isomorphic to an algebra obtained from 
  a  basic  hereditary algebra $\bH$ by a suitable sequence of gluings and blowings up 
  components. 
 \end{theorem}
  
  In other words, a basic nodal algebra can be given by a quiver and a symmetric relation
  $\sim$ on the set of its vertices such that $\#\setsuch{j}{i\sim j}\le1$ for each vertex $i$.

\begin{proof}
   Proposition~\ref{properties} implies that if an algebra $\bA$ is obtained from 
  a  basic  hereditary algebra $\bH$ by a suitable sequence of gluings and blowings up, then 
  it is nodal. To prove the converse, we use a lemma about semisimple 
  algebras.
  
   \begin{lemma}\label{semisimple} 
  Let $\tS=\prod_{i=1}^m\tS_i$ be a semisimple  algebra, where $\tS_i\simeq\Mat(d_i,\aK)$
  are its simple components, $\bS=\prod_{k=1}^r\bS_k$ be its subalgebra such that 
  $\bS_k\simeq\aK$ and $\len_\bS(\tS\*_\bS\bS_k)\le2$ for all $1\le k\le r$. 
  Then, for each $1\le k\le r$ 
  \begin{enumerate}
  \item either $\bS_k=\tS_i$ for some $i$,
  \item or $\bS_k\sb\tS_i\xx\tS_j$ for some $i\ne j$ such that $\tS_i\simeq\tS_j\simeq\aK$
  and $\bS_k\simeq\aK$ embeds into $\tS_i\xx\tS_j\simeq\aK\xx\aK$ diagonally,
  \item or there is another index $q\ne k$ such that $\bS_k\xx\bS_q\sb\tS_i$
  for some $i$, $\tS_i\simeq\Mat(2,\aK)$ and this isomorphism can be so chosen that
  $\bS_k\xx\bS_q$ embeds into $\tS_i$ as the subalgebra of diagonal matrices.
  \end{enumerate}
 \end{lemma}
 \begin{proof}
 Denote $L_{ik}=\tS_i\*_\bS\bS_k$. Certainly $L_{ik}\ne0$ \iff the projection of $\bS_k$
 onto $\tS_i$ is non-zero. Since $L_k=\tS\*_\bS\bS_k=\bop_{i=1}^mL_{ik}$, there are at most
 two indices $i$ such that $L_{ik}\ne0$. Therefore either $\bS_k\sbe\tS_i$ for some $i$ or
 $\bS_k\sbe\tS_i\xx\tS_j$ for some $i\ne j$ and both $L_{ik}$ and $L_{jk}$ are non-zero. 
 Note that $\dim_\aK L_{ik}\ge d_i$ and $\dim_\aK L_k\le2$.
 So in the latter case $\tS_i\simeq\tS_j\simeq\aK$. Obviously,
 $\aK$ can embed into $\aK\xx\aK$ only diagonally.
 
 Suppose that $\bS_k\sbe\tS_i$ but $\bS_k\ne\tS_i$. Then $d_i=2$, so $\tS_i\simeq\Mat(2,\aK)$.
 Then the unique simple $\tS_i$-module is $2$-dimensional. If $\bS_k$ is the only simple $\bS$-module
 such that $\bS_k\sb\tS_i$, then it embeds into $\tS_i$ as the subalgebra of scalar matrices,
 thus $L_{ik}\simeq\tS_i$ is of dimension $4$, which is impossible. 
 Hence there is another index $q\ne k$ such 
 that $\bS_q\sb\tS_i$. Then the image of $\bS_k\xx\bS_q\simeq\aK^2$ in $\Mat(2,\aK)$ 
 is conjugate to the subalgebra of diagonal matrices \cite[Chapter II, Exercise 2]{dk}. 
 \end{proof}
 
  Let now $\bA$ be a nodal algebra related to a hereditary algebra $\tH$, 
  $\tS=\tH/\rad\tH$ and $\oA=\bA/\rad\bA$. We denote by $\bH$ 
  the basic algebra of $\tH$ \cite[Section III.5]{dk} and for each simple component 
  $\tS_i$ of $\tS$ we denote by $\oH_i$ the corresponding simple components of 
  $\oH=\bH/\rad\bH$. We can apply Lemma~\ref{semisimple} to the algebra 
  $\tS=\tH/\rad\tH$ and its subalgebra $\oA=\bA/\rad\bA$. Let $(i_1,j_1),\dots,(i_r,j_r)$
  be all indices such that the products $\tS_{i_k}\xx\tS_{j_k}$ occur as in the case (2)
  of the Lemma, while $i_{r+1},\dots,i_{r+s}$ be all indices such that $\tS_{i_k}$ occur
  in the case (3). Then it is evident that $\bA$ is obtained from $\bH$ by the suitable
  sequence of gluings and blowings up defined by the sequence of indices
  $(\lst i{r+s},\lst jr)$.
  \end{proof}
  
  Since the construction of gluing and blowing up is left--right symmetric, we get the
  following corollary.
  
  \begin{corol}\label{opposite} 
   If an algebra $\bA$ is nodal, so is its opposite algebra. In particular, in the 
   Definition~\ref{nodal} one can replace the condition $(3)$ by the condition $(3')$
   from Remark~\ref{remark}.
  \end{corol}
    
  Thus, to define a basic nodal algebra, we have to define a quiver $\Ga$ and a sequence of
  its vertices $(\lst i{r+s},\lst jr)$. Actually, one can easily describe the resulting
  algebra $\bA$ by its quiver and relations. Namely, we must proceed as follows:
 \begin{enumerate}
   \item For each $1\le k\le r$
   \begin{enumerate}
    \item we glue the vertices $i_k$ and $j_k$ keeping all arrows 
   starting and ending at these vertices;
    \item if an arrow $\al$ starts at the vertex $i_k$ (or $j_k$) and an arrow $\be$
    ends at the vertex $j_k$ (respectively $i_k$), we impose the relation $\al\be=0$.
   \end{enumerate} 
   \item For each $r<k\le r+s$
	\begin{enumerate}
      \item  we replace each vertex $i_k$ by two vertices $i_k'$ and $i_k''$;
      \item  we replace each arrow $\al:j\to i_k$ by two arrows $\al':j\to i_k'$
      and $\al'':j\to i_k''$;
      \item  we replace each arrow $\be:i_k\to j$ by two arrows $\be':i'_k\to j$
      and $\be'':i_k''\to j$;
      \item if an arrow $\be$ starts at the vertex $i_k$ and an arrow $\al$ ends at this 
      vertex, we impose the relation $\be'\al'=\be''\al''$.
   \end{enumerate}   
 \end{enumerate}  
  We say that $\bA$ is a \emph{nodal algebra of type} $\Ga$. In particular, if $\Ga$ is a
  Dynkin quiver of type $\rA$ or a Euclidean quiver of type $\ti\rA$, we say that $\bA$ is a 
  \emph{nodal algebra of type $\rA$}. 
  
 To define a nodal algebra which is not necessarily
  basic, we also have to prescribe the multiples $l_i$ for each vertex $i$ so that 
  $l_{i_k}=l_{j_k}$ for $1\le k\le r$. 
  
  In what follows we often present a basic nodal algebra by the quiver $\Ga$, just marking
  the vertices $\lst ir,\lst jr$ by bullets, with the indices $1,2,\dots,r$, and marking
  the vertices $i_{r+1},\dots,i_{r+s}$ by circles. For instance:
  \[
   \xymatrix{  \ccdot \ar[r]^{\al_1} & \bullet1 \ar[r]^{\al_2} & 
    \ccdot \ar[r]^{\al_3} \ar[d]^{\al_5} & \bullet1 &
    \bullet2 \ar[l]_{\al_4} \\
    && \circ3 \ar[d]^{\al_6} \\ && \bullet2 }
  \]
  In this example the resulting nodal algebra $\bA$ is given by the quiver with relations
  \[
   \xymatrix{&& \circ3' \ar[dr]^{\al'_6} \\
   \ccdot \ar[r]^{\al_1} & \bullet1 \ar@<.5ex>[r]^{\al_2} &
   \ccdot \ar@<.5ex>[l]^(.2){\al_3} \ar[u]^{\al'_5} \ar[dr]_(.6){\al''_5} &
   \bullet2 \ar@/^1.5pc/[ll]_<<{\al_4} & {}\quad
   \begin{minipage}{4cm}
   $\al_2\al_3=\al_2\al_4=0$\newline $\al_4\al'_6=\al_4\al''_6=0$\newline
    $\al'_6\al'_5=\al''_6\al''_5$
   \end{minipage}
      \\ &&& \circ3'' \ar[u]_{\al''_6} }
  \]
         
  \section{Inessential gluings}
  \label{sec2} 

  In this section we study one type of gluing which never implies the representation
  type.
  
  \begin{defin}\label{inessential} 
  Let a basic  algebra $\bB$ is given by a quiver $\Ga$ with relations and an algebra
  $\bA$ is obtained from $\bB$ by gluing the components corresponding to the vertices
  $i$ and $j$ such that there are no arrows ending at $i$ and no arrows starting at $j$.
  Then we say that this gluing is \emph{inessential}.
  \end{defin}
  
  It turns out that the categories $\bA\md$ and $\bB\md$ are ``almost the same.''
  
  \begin{theorem}\label{effect} 
   Under the conditions of Definition~\ref{inessential}, there is an equivalence of the
   categories $\bB\md/\gnr{\oB_i,\oB_j}$ and $\bA\md/\gnr{\oA_{ij}}$, where $\kC/\gnr{\dM}$
   denotes the quotient category of $\kC$ modulo the ideal of morphisms that factor through
   direct sums of objects from the set $\dM$.
  \end{theorem}
   \begin{proof}
    We identify $\bB$-modules and $\bA$-modules with the representations of the
    corresponding quivers with relations. Recall that the quiver of $\bA$ is obtained
    from that of $\bB$ by gluing the vertices $i$ and $j$ into one vertex $(ij)$.
    For a $\bB$-module $M$ denote by $\bF M$ the $\bA$-module such that
   \begin{equation}\label{FM} 
   \begin{split}
    \bF M(k)&=M(k) \text{ for any vertex } k\ne(ij),\\
    \bF M(ij)&=M(i)\+M(j),\\
    \bF M(\ga)&=M(\ga) \text{ if } \ga\notin\Ar(ij),\\
    \bF M(\al)&=\mtr{M(\al)&0} \text{ if } \al\in\Ar(i)\=\Ar(j),\\
    \bF M(\be)&=\mtr{0\\M(\be)} \text{ if } \be\in\Ar(j)\=\Ar(i)\\
    \bF M(\al)&=\mtr{0&0\\M(\al)&0} \text{ if } \al:i\to j. 
  \end{split}
  \end{equation}
   If $f:M\to M'$ is a homomorphism of $\bB$-modules, we define the homomorphism 
   $\bF f:\bF M\to \bF M'$ setting
   \begin{align*}
   \bF f(k)&=f(k)\, \text{ if } k\ne (ij),\\
   \bF f(ij)&=\begin{pmatrix}
   f(i) & 0 \\ 
   0 & f(j)
   \end{pmatrix} 
   \end{align*}
   Thus we obtain a functor $\bF:\bB\md\to\bA\md$. Obviously, $\bF\oB_i=\bF\oB_j=\oA_{ij}$,
   so $\bF$ induced a functor $\fF:\bB\md/\gnr{\oB_i,\oB_j}\to\bA\md/\gnr{\oA_{ij}}$. 
      
   Let now $N$ be an $\bA$-module. We define the $\bB$-module $\bG N$ as follows:
   \begin{align*}
   \bG N(k)&=N(k) \text{ if } k\notin\{i,j\},\\
   \bG N(i)&=N(ij)/N_0(ij),  \text{ where } N_0(ij)=\bap_{\al^-=(ij)}\Ker N(\al),\\
   \bG N(j)&=\sum_{\be^+=(ij)}\im N(\be),\\
   \bG N(\be)&=N(\be)  \text{ if } \be\notin \Ar(i),\\
   \bG N(\al)&  \text{ is the induced map } \bG N(i)\to \bG N(k)  \text{ if } 
   				\al:i\to k.
   \end{align*}
   Note that if $\be^+=j$, then $\im N(\be)\sbe \bG N(j)$. If $g:N\to N'$ is a homomorphism
   of $\bA$-modules, then $g(ij)(\bG N(j))\sbe \bG N'(j)$ and $g(ij)(N_0(ij))\sbe N'_0(ij)$. So
   we define the homomorphism $\bG g:\bG N\to \bG N'$ setting
   \begin{align*}
   \bG g(k)&=g(k)\,  \text{ if } k\ne i,\\
   \bG g(i) & \text{ is the map $\bG N(i)\to\bG N'(i)$ induced by } g(ij),\\
   \bG g(j)&  \text{ is the resriction of } g(ij)  \text{ onto } \bG N(j).
   \end{align*}
   Thus we obtain a functor $\bG:\bA\md\to\bB\md$. Since $\bG\oA_{ij}=0$, it induces
   a functor $\fG:\bA\md/\gnr{\oA_{ij}}\to\bB\md/\gnr{\oB_i,\oB_j}$. Suppose that $\bG g=0$.
   It means that $g(k)=0$ for $k\ne(ij)$, $\im g(ij)\sbe\bap_{\al^-=(ij)}\Ker N'(\al)$ and
   $\Ker g(ij)\spe\sum_{\be^+=(ij)}\im N(\be)$. So $g(ij)$ induces the map 
   $$
   \bar g:N(j)/\sum_{\be^+=j}\im N(\be)\to N'(ij)
   $$ 
   with 
   $\im\bar g\sbe\bap_{\al^-=(ij)}\Ker N'(\al)$. So $g=g''g'$, where
   \begin{align*}
   g':& N\to \ol N  \text{ and } g'':\ol N\to N',\\
   \ol N(k)&=0  \text{ if } k\ne(ij),\\
   \ol N(ij)&= N(j)/\sum_{\be^+=j}\im N(\be),\\
   g'(k)&=g''(k)=0  \text{ if } k\ne(ij),\\
   g'(ij)&  \text{ is the natural surjection } N(ij)\to\ol N(ij),\\
   g''(ij)&=\bar g.
   \end{align*}
   Obviously, $\ol N\simeq m\oA_{ij}$ for some $m$,
   so $\Ker\bG$ is just the ideal $\gnr{\oA_{ij}}$. 
   
   By the construction,
   \begin{align*}
    \bG\bF M(i)&=M(i)/\bap_{\al^-=i}\Ker\al,\\
   \bG\bF M(j)&=\sum_{\be^+=j}\im\be,\\
  \bF\bG N(ij)&=N(ij)/\bap_{\al^-=i}\Ker\al\+\sum_{\be^+=(ij)}\im N(\be).
   \end{align*}   
   So we fix
   \begin{itemize}
   \item[] for every $\bB$-module $M$ a retraction $\rho_M:M(j)\to\sum_{\be^+=j}\im\be$,
   \item[] for every $\bA$-module $N$ a retraction $\rho_N:N(ij)\to\sum_{\be^+=(ij)}\im\be$
   \end{itemize}
   and define the morphisms of functors
   \begin{align*}
   \phi:&\ \id_{\bB\md}\to \bG\bF \text{ such that}\\
     & \phi_M(k)=\id_{M(k)}  \text{ if } k\notin\{i,j\},\\
     & \phi_M(j)=\rho_M,\\
     & \phi_M(i)  \text{ is the natural surjection } M(i)\to\bG\bF M(i),
     \intertext{and}
   \psi: &\ \id_{\bA\md}\to \bF\bG \text{ such that}\\
     & \psi_N(k)=\id_{N(k)}  \text{ if } k\ne(ij),\\
     & \psi_N(ij)=\rho_N.
   \end{align*}
   Evidently, if $M$ has no direct summands $\oB_i$ and $\oB_j$, then $\phi_M$ is an 
   isomorphism. Also if $N$ has no direct summands $\oA_{ij}$, then $\psi_N$ is an
   isomorphism. Therefore, $\fG$ and $\fF$ are mutually quasi-inverse, defining an
   equivalence of the  categories $\bB\md/\gnr{\oB_i,\oB_j}$ and $\bA\md/\gnr{\oA_{ij}}$.
   \end{proof}
  
  \section{Gentle and skewed-gentle case}
  \label{sec3} 
  
  In what follows we only consider non-hereditary nodal algebras, since the 
  representation types of hereditary algebras are well-known. Evidently, blowing up 
  a vertex $i$ such that there are no arrows starting at $i$ or no arrows ending at $i$, 
  applied to a hereditary algebra, gives a hereditary algebra. The same happens if we glue
  vertices $i$ and $j$ such that there are no arrows starting at these vertices or no arrows 
  ending at them. 
  
  Recall that a basic (\fdd) algebra $\bA$ is said to be \emph{gentle} if it is given by a
  quiver $\Ga$ with relations $R$ such that 
\begin{enumerate}
  \item for every vertex $i\in \Ga$, there are at most two arrows starting at $i$ and at most
  two arrows ending at $i$;
  \item all relations in $R$ are of the form $\al\be$ for some arrows $\al,\be$;
  \item if there are two arrows $\al_1,\al_2$ starting at $i$, then, for each arrow $\be$
  ending at $i$, either $\al_1\be\in R$ or $\al_2\be\in R$, but not both;
  \item if there are two arrows $\be_1,\be_2$ ending at $i$, then, for each arrow $\al$
  starting at $i$, either $\al\be_1\in R$ or $\al\be_2\in R$, but not both.
\end{enumerate}
  A basic algebra $\bA$ is said to be \emph{skewed-gentle} if it can be obtained from a gentle algebra
  $\bB$ by blowing up some vertices $i$ such that there is at most one arrow $\al$ starting at $i$,
  at most one arrow $\be$ ending at $i$ and if both exist then $\al\be\notin R$.\!%
\footnote{\,The original definition of skew-gentle algebras in \cite{gp} as well as that 
in \cite{bek} differ from ours, but one can easily see that all of them are equivalent.}

 It is well-known that gentle and skewed-gentle algebras are tame, and even derived tame (i.e.
 their derived categories of \fdd\ modules are also tame). Skowronski and Waschb\"usch \cite{sw}
 proved a criterion of representation finiteness for biserial algebras, the class containing,
 in particular, all gentle algebras. We give a complete description 
 of nodal algebras which are gentle or skewed-gentle. 
 
 \begin{theorem}\label{gentle}  
  A nodal algebra $\bA$ is skewed-gentle \iff it is obtained from a 
  direct product of hereditary algebras of type $\rA$ or $\tA$ by a suitable sequence of 
  gluings and blowings up defined by a sequence of vertices such that, for each of them, 
  there is at most one arrow starting and at most one arrow ending at this vertex. 
  It is gentle if and only if, moreover, it is obtained using only gluings.
 \end{theorem}
 \begin{proof} 
  If $\bA$ is related to a hereditary algebra $\bH$ such that its quiver is not a 
  disjoint union of quivers of type $\rA$ or $\tA$, 
  there is a vertex $i$ in the quiver of $\bH$ such that $\Ar(i)$ has at least $3$ arrows.
  The same is then true for the quiver of $\bA$. Moreover, there are no relations containing 
  more that one of these arrows, which is impossible in a gentle or skewed-gentle algebra. 
  
  So we can suppose that the quiver of $\bH$ is a disjoint union of quivers of type 
  $\rA$ or $\tA$. Let $\bA$ is obtained from $\bH$
  by a suitable sequence of gluings and blowings up defined by a sequence of vertices 
  $\lst i{r+s},\lst jr$. Suppose that there is an index $1\le k\le r+s$ such that there are two
  arrows $\al_1,\al_2$ ending at $i_k$ (the case of two starting arrows is analogous). 
  If $k\le r$ and there is an arrow ending at $j_k$, there are $3$ vertices ending at the vertex
  $(ij)$ in the quiver of $\bA$, neither two of them occurring in the same relation, which is
  impossible in gentle or skewed-gentle case. 
  If there is an arrow $\be$ starting at $j_k$, it occurs in two zero relations 
  $\be\al_1=\be\al_2=0$, which is also impossible.
  
  Finally, if we apply blowing up, we obtain three arrows incident to a vertex without zero relations
  between them which is impossible in a gentle algebra. Thus the conditions of the theorem are 
  necessary.
  
  On the contrary, let $\bH$ be a hereditary algebra and its quiver is a disjoint union of
  quivers of type $\rA$ or $\tA$, $i_1\ne i_2$ be vertices 
  of this quiver such that there is a unique arrow $\al_k$ starting at $i_k$ as well as a unique
  arrow $\be_k$ ending at $i_k$ ($k=1,2$). Then gluing of vertices $i_1,i_2$ gives a vertex
  $i=(i_1i_2)$ in the quiver of the obtained algebra, two arrows $\al_k$ starting at $i$ and
  two arrows $\be_k$ ending at $i$ ($i=1,2$) satisfying relations $\al_1\be_2=\al_2\be_1=0$. 
  Therefore, gluing such vertices give a gentle algebra. Afterwards, blowing up vertices $j$
  such that there is one arrow $\al$ starting at $j$ and one arrow $\be$ ending at it gives a 
  skewed-gentle algebra since $\al\be\ne0$ in $\bH$.
 \end{proof}

\section{Exceptional algebras}
\label{sec4}

 We consider some more algebras obtained from hereditary algebras of type $\rA$.

\begin{defin}
 Let $\bH$ be a basic hereditary algebra with a quiver $\Ga$.
 \begin{enumerate}
 \item   We call a pair of vertices $(i,j)$ of the quiver $\Ga$ 
 \emph{exceptional} if they are contained in a full subquiver of the shape 
\begin{align}
&\xymatrix{ \ccdot \ar[r]^{\be} & {\str i\ccdot} &\ccdot \ar[l]_{\al_1} \ar@{-}[r]^{\al_2}& 
 \dots & \ccdot \ar@{-}[l]_{\al_{n-1}} & {\str j\ccdot}\ar[l]_{\al_n}  & \ccdot \ar[l]_{\ga} }
 \label{exc1} \\
 \intertext{or}
&\xymatrix{ \ccdot & {\str i\ccdot} \ar[l]_{\be} \ar[r]^{\al_1} &\ccdot \ar@{-}[r]^{\al_2}& 
 \dots & \ccdot \ar@{-}[l]_{\al_{n-1}}\ar[r]^{\al_n} & {\str j\ccdot} \ar[r]^{\ga} & \ccdot }
 \label{exc2} 
\end{align}
 where the orientation of the arrows $\al_2,\dots,\al_{n-1}$ is arbitrary. Possibly $n=1$,
 then $\al_1=\al_n:j\to i$ in case \eqref{exc1} and $\al_1=\al_n:i\to j$ in case \eqref{exc2}.

 \item We call gluing of an exceptional pair of vertices \emph{exceptional gluing}. 
 
 \item 
 A nodal algebra is said to be \emph{exceptional} if it is obtained from a 
 hereditary algebra of type $\rA$ by a suitable sequence of gluings consisting of one exceptional 
 gluing and, maybe, several inessential gluings.
 \end{enumerate}
 \end{defin} 
 
 Recall that inessential gluing does not imply the representations type of an algebra.
 So we need not take them into account only considering exceptional algebras obtained
 by a unique exceptional gluing. Note that such gluing results in the algebra $\bA$ 
 given by the quiver with relations
 \begin{align}
&\xymatrix{  && \ccdot \ar[dr]^{\al_1}\ar@{.}[rr] && 
 \ccdot &&&\hspace*{-6em} \al_n\al_1=\al_n\be=0\\
\ccdot \ar@{-}[r]^{\be_m}&\dots  &\ar@{-}[l]_{\be_1}
\ccdot \ar[r]^{\be} & {\str {(ij)}\ccdot} \ar[ur]^{\al_n}  & \ccdot \ar[l]_{\ga} 
 \ar@{-}[r]^{\ga_1}& \dots & \ccdot \ar@{-}[l]_{\ga_l} }
 \label{case1} \\
 \intertext{in case \eqref{exc1} or}
&\xymatrix{ && \ccdot \ar@{.}[rr] && \ccdot \ar[dl]_{\al_n} &&&
 \hspace*{-6em}\al_1\al_n=\be\al_n=0\\
\ccdot \ar@{-}[r]^{\be_m}&\dots  &\ar@{-}[l]_{\be_1}
\ccdot & {\str {(ij)}\ccdot} \ar[l]_{\be} \ar[ul]_{\al_1} \ar[r]^{\ga} & \ccdot 
 \ar@{-}[r]^{\ga_1}& \dots  &\ccdot \ar@{-}[l]_{\ga_l} }
 \label{case2} 
 \end{align}
 in case~\eqref{exc2}. The dotted line consists of the arrows 
 $\al_2,\dots,\al_{n-1}$; if $n=1$, we get a loop $\al$ at the vertex $(ij)$ with the
 relations $\al^2=0$ and, respectively, $\al\be=0$ or $\be\al=0$. We say that $\bA$ is an 
 $(n,m,l)$-\emph{exceptional algebra}. 

 We determine representation types of exceptional algebras. 
 
 \begin{theorem}\label{except} 
  An $(n,m,l)$-exceptional algebra is
  \begin{enumerate}
  \item representation finite in cases:
    \begin{enumerate}
    \item $m=l=0$,
    \item $l=0,\,m=1,\,n\le 3$,
    \item $l=0,2\le m\le3,\,n=1$,
    \item $m=0,\,l=1,\,n\le2$.
÷    \end{enumerate}
  \item tame in cases:
    \begin{enumerate}
    \item $l=0,\,m=1,\,n=4$,
    \item $l=0,\,m=2,\,n=2$,
    \item $l=0,\,m=4,\,n=1$,
    \item $m=0,\,l=1,\,n=3$.
    \end{enumerate}
  \item wild in all other cases.
  \end{enumerate}
 \end{theorem}
 \begin{proof}
  We consider an algebra $\bA$ given by the quiver with relations~\eqref{case2}.  Denote by $\bC$ 
  the algebra given by the quiver with relations
  \[
   \xymatrix@R=1ex@C=2em{ \ccdot \ar@{.}[rr] && \ccdot \ar[dl]^{\al_n}
   				&       \al_1\al_n=0 \\
   				& \bullet \ar[ul]^{\al_1}  }
  \]
  (the bullet shows the vertex $(ij)$).
  It is obtained from the path algebra of the quiver
  \[
  \Ga_n= \xymatrix{\str0\ccdot \ar[r]^{\al_1} & \str1\ccdot \ar@{-}[r]^{\al_2} &\dots &
   			\str{n-1}\ccdot \ar@{-}[l]_{\al_{n-1}} \ar[r]^{\al_n} & \str n\ccdot }
  \] 
  by gluing vertices $0$ and $n$. This gluing is inessential, so we can use Theorem~\ref{effect}.
  We are interested in the indecomposable representations of $\bC$ that are non-zero at the 
  vertex $\bullet$\,. We denote by $\dL$ the set of such representations.
  They arise from the representations of the quiver $\Ga_n$ that are non-zero at the vertex 
  $0$ or $n$. Such representations are $\ti L_i$ and $\ti L'_i$ ($0\le i\le n$), where
  \begin{align*}
  \ti L_i(k)&\begin{cases}
    \aK &\text{ if } k\le i,\\
    0 &\text{ if } k>i;
  \end{cases} \\
  \ti L'_i(k)&=\begin{cases}
    \aK &\text{ if } k\ge i,\\
    0 & \text{ if } k<i.
  \end{cases}
  \end{align*}
  We denote by $L_i$ and $L'_i$ respectively the representations  $\bF\ti L_i$ and $\bF\ti L'_i$ 
  (see page \pageref{FM}, formulae \eqref{FM}). Obviously $L_n=L'_0$, $L_0=L'_n=\ol\bC_{\bullet}$ and
  $\dim_\aK L_i(\bullet)=1$ for $i\ne n$ as well as $\dim_\aK L'_i(\bullet)=1$ for $i\ne0$, while
  $\dim_\aK L_n(\bullet)=2$. We denote by $e_i\ (0\le i<n)$ and $e'_j\ (0<i<n)$ basic vectors 
  respectively of $L_i$ and $L'_j$, and by $e_n,e'_n$ basic vectors of $L_n=L'_0$ such that 
  $e'_n\in\im L_n(\al_n)$; then $L_n(\al_1)(e'_n)=0$, while $L_n(\al_1)(e_n)\ne0$. We consider
  the set $\bE=\set{\lsto en}$ and the relation $\prec$ on $\bE$, where $u\prec v$ means that   
  there is a homomorphism $f$ such that $f(u)=v$. From the well-known (and elementary) description 
  of representations of the quiver $\Ga$ and Theorem~\ref{effect} it follows that $\prec$ is a linear
  order and $e_i\prec e_0\prec e'_j$ for all $i,j$. Let $\lsto u{2n}$ be such a numeration of
  the elements of $\bE$ that $u_i\prec u_j$ \iff $i\le j$ (then $u_n=e_0$), and $e_n=u_k,e'_n=u_l$ 
  ($k<n<l$). Note also that if $f\in\End L_n$, the matrix $f(\bullet)$ in the basis $e_n,e'_n$ 
  is of the  shape $\mtr{\la&0\\ \mu&\la}$.
  
  Let $M$ be an $\bA$-module, $\oM$ be its restriction onto $\bC$. 
  Then $\oM\simeq\bop_im_iL_i\+\bop_jm'_jL'_j$, where
  $M_i\simeq $ and $M'_j\simeq $. Respectively $M(\bullet)=\bop_{i=1}^{2n+1}U_i$,
  where $U_i$ is generated by the images of the vectors $u_i$. Note that, for $i>n$, $u_i=e'_j$
  for some $j$, so $M(\be)U_i=0$. Therefore, the matrices $M(\be)$ and $M(\ga)$ shall be 
  considered as block matrices
  \begin{equation}
  \begin{split}
  M(\be)&=\mtr{ B_0 & B_1& \dots & B_n & 0 & \dots & 0},\\
  M(\ga)&=\mtr{ C_0 & C_1 & \dots & C_n & C_{n+1} & \dots & C_{2n}},
  \end{split}
  \end{equation}
  where the matrices $B_i,C_i$ correspond to the summands $U_i$. If $f\in\hom_\bA(M,N)$, 
  then $f(\bullet)$ is a block lower triangular matrix $(f_{ij})$, where $f_{ij}:U_j\to U_i$
  and $f_{ij}=0$ if $i<j$. Moreover, the non-zero blocks can be arbitrary with the only
  condition that $f_{kk}=f_{ll}$. Hence given any matrix $f(\bullet)$ with this condition and 
  invertible diagonal blocks $f_{ii}$, we can construct a module $N$ isomorphic to $M$ just 
  setting $N(\be)=M(\be)f(\bullet),\ N(\ga)=M(\ga)f(\bullet)$. Then one can easily transform 
  the matrix $M(\ga)$ so that there is at most one non-zero element (equal $1$) in every row and 
  in every column, if $i\notin\{k,l\}$, the non-zero rows of $C_i$ have the form $\mtr{I&0}$ 
  and the non-zero rows of the matrix $(C_k\,|\,C_l)$ are of the form
  \[
  \left[
  \begin{array}{cccc|cccc}
   0&0&0&0 & I&0&0&0 \\	0&0&0&0 & 0&I&0&0 \\
   0&0&I&0 & 0&0&0&0 \\ I&0&0&0 & 0&0&0&0 
  \end{array}
  \right]
  \]
  We subdivide the columns of the matrices $B_i$ respectively to this subdivision of $C_i$.
  It gives $2(n+2)$ new blocks $\tB_s (1\le s\le 2(n+2))$. Namely, the blocks $\tB_s$ with $s$
  odd correspond to the non-zero blocks of the matrices $C_i$ and those with $s$ even correspond
  to zero columns of $C_i$. Two extra blocks come from the subdivision of $C_k$ into
  $4$ vertical stripes. We also subdivide the blocks 
  $f_{ij}$ of the matrix $f(\bullet)$ in the analogous  way. 
  From now on we only consider such representations that the matrix $M(\ga)$ is of the form reduced 
  in this way. One can easily check that it imposes the restrictions on the matrix 
  $f(\bullet)$ so that the new blocks $\tF_{st}$ obtained from $f_{ij}$ with 
  $0\le i,j\le n$ can only be non-zero (and then arbitrary) \iff $s>t$ and, moreover, 
  either $t$ is odd or $s$ is even. It means that these 
  new blocks can be considered as a representation of the \emph{poset} 
  (partially ordered set) $\dS_{n+2}$: 
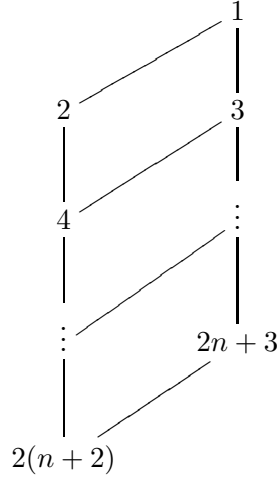
\label{ds}
  \[
   \xymatrix{ & 1 \ar@{-}[dl] \ar@{-}[d] \\ 2 \ar@{-}[d] &   3 \ar@{-}[dl] \ar@{-}[d] \\ 
   4 \ar@{-}[d] &	 \vdots \ar@{-}[dl] \ar@{-}[d] \\ \vdots \ar@{-}[d]  & 2n+3 \ar@{-}[dl]  \\ 
   2(n+2) } 
  \]
  (in the sense of \cite{nr}). It is well-known \cite{nr} that $\dS_{n+2}$ has finitely 
  many indecomposable representations. It implies that $\bA$ is representation finite if $m=l=0$.
  
  If $l=1$, let $\ga:j\to j_1,\,\ga_1:j_2\to j_1$ (the case $\ga_1:j_1\to j_2$ is analogous). 
  We can suppose that the matrix $M(\ga_1)$ is of the form $\mtr{I&0\\0&0}$. Then the rows
  of all matrices $C_i$ shall be subdivided respectively to this division of $M(\ga_1)$:
  $C_i=\mtr{C_{i1}\\C_{i0}}$. Moreover, if $f$ is a homomorphism of such representations, then 
  $f(j_1)=\mtr{c_1&c_2\\0&c_3}$. Quite analogously to the previous considerations one can see that,
  reducing $M(\ga)$ to a canonical form, we subdivide the columns of $B_i$ so that resulting problem
  becomes that of representations of the poset $\dC'_{n+3}$:
  
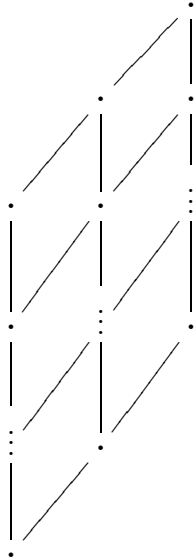
\[
    \xymatrix{ && \ccdot \ar@{-}[d] \ar@{-}[dl] \\
    			& \ccdot \ar@{-}[d] \ar@{-}[dl] & \ccdot \ar@{-}[d] \ar@{-}[dl] \\
    	 \ccdot \ar@{-}[d] & \ccdot \ar@{-}[d] \ar@{-}[dl] &\vdots \ar@{-}[d] \ar@{-}[dl] \\
		 \ccdot \ar@{-}[d] & \vdots \ar@{-}[d] \ar@{-}[dl] &\ccdot  \ar@{-}[dl] \\			 						\vdots \ar@{-}[d] &\ccdot\ar@{-}[dl]  \\
		 \ccdot }
  \]
  ($n+3$ points in each column). The results of \cite{kl,na} imply that this problem is finite
  if $n\le 2$, tame if $n=3$ and wild if $n>3$. Therefore, so is the algebra $\bA$ if $m=0$ and
  $l=1$. 
  
  If $l>1$ then after a reduction of the matrices $M(\ga_2),M(\ga_1)$ and $M(\ga)$ we obtain
  for $M(\be)$ the problem of the representations of the poset $\dS''_{n+4}$ analogous to 
  $\dS$ and $\dS'$ but with $4$ columns and $n+4$ point in every column. This problem is
  wild \cite{na}, hence the algebra $\bA$ is also wild. If both $l>0$ and $m>0$, analogous 
  consideration shows that if we reduce the matrix $M(\be_1)$ to the form $\mtr{I&0\\0&0}$, 
  the rows of the matrix $M(\be)$ will also be subdivided, so that we obtain the problem on
  representations of a \emph{pair of posets} \cite{kl}, one of them being $\dS'_{n+3}$ and 
  the other being linear ordered with $2$ elements. It is known from \cite{kl1} that this 
  problem is wild, so the algebra $\bA$ is wild as well.
  
  Let now $l=0$ and fix the subdivision of columns of $M(\be)$ described by the poset 
  $\dS_{n+2}$ as above. If we reduce the matrices $M(\be_i)$, which form representations of the
  quiver of type $\rA_m$, the rows of $M(\be)$ will be subdivided so that as a result we obtain
  representations of the pair of posets, one of them being $\dS_{n+2}$ and the other being
  linear ordered with $m+1$ elements. The results of \cite{kl,kl1} imply that this problem is
  representation finite if either $m=1,\,n\le3$ or $2\le m\le3,\,n=1$, tame if either 
  $m=1,\,n=4$, or $m=2,\,n=2$, or $m=4,\,n=1$. In all other cases it is wild. Therefore, the 
  same is true for the algebra $\bA$, which accomplishes the proof.
  \end{proof}
  
  We use one more class of algebras.
  
  \begin{defin}\label{super} 
   A nodal algebra is said to be \emph{super-exceptional} if it is obtained from an algebra of the 
   form \eqref{case1} or \eqref{case2} with $n=3$ by gluing the ends of the arrow $\al_2$ in the case 
   when such gluing is not inessential, and, maybe, several inessential gluings.
  \end{defin}
  
  Obviously, we only have to consider super-exceptional algebras obtained without inessential gluings.
  Using \cite[Theorem 2.3]{kl1}, one easily gets the following result.
  
 \begin{prop}\label{sup-type} 
   A super-exceptional algebra is 
 \begin{enumerate}
 \item representation finite if $m=l=0$,
 \item tame if $m+l=1$,
 \item wild if $m+l>1$.
 \end{enumerate}
   \end{prop}  
  
  \section{Final result}
  \label{sec5}
  
  Now we can completely describe representation types of nodal algebras of type $\rA$.

 \begin{defin}\label{quasi} 
  \begin{enumerate}
  \item   We call an algebra $\bA$ \emph{quasi-gentle} if it can be obtained from a gentle or
  skewed-gentle algebra by a suitable sequence of inessential gluings.
  \item  We call an algebra \emph{good exceptional} (\emph{good super-exceptional}) if it is 
  exceptional (respectively, super-exceptional) and not wild.
  \end{enumerate}
 \end{defin}

  Theorem~\ref{except} and Proposition~\ref{sup-type} give a description of good exceptional 
  and super-exceptional algebras.
  
\begin{theorem}\label{main} 
 A non-hereditary nodal algebra of type $\rA$ is representation finite or 
 tame \iff it is either quasi-gentle, or good exceptional, or good super-exceptional. 
 In other cases it is wild.
\end{theorem}

 Before proving this theorem, we show that gluing or blowing up cannot ``improve'' representation 
 type.
 
 \begin{prop}\label{embed} 
 Let an algebra $\bA$ be obtained from $\bB$ by gluing or blowing up. Then there is an exact functor
 $\bF:\bB\md\to\bA\md$ such that $\bF M\simeq\bF M'$ \iff $M\simeq M'$ or, in case of gluing  
 vertices $i$ and $j$, $M$ and $M'$ only differ by trivial direct summands at these vertices. 
 \end{prop}
 \begin{proof}
  Let $\bA$ is obtained by blowing up a vertex $i$. We suppose that there are no loops at this
  vertex. The case when there are such loops can be treated analogously but the formulae become
  more cumbersome. Note that in the further consideration we do not need this case. For a $\bB$-module 
  $M$ set $\bF M(k)=M(k)$ if $k\ne i$, $\bF M(i')=\bF M(i'')=M(i)$, $\bF M(\al)=M(\al)$ if 
  $\al\notin\Ar(i)$ and $\bF M(\al')=\bF M(\al'')=M(\al)$ if $\al\in\Ar(i)$. 
  If $f:M\to M'$, set
  $\bF f(k)=f(k)$ if $k\ne i$ and $\bF f(i')=\bF f(i'')=f(i)$. It gives an exact functor 
  $\bF:\bB\md\to\bA\md$. Conversely, if $N$ is an $\bA$-module, set $\bG N(k)=N(k)$ if $k\ne i$
  and $\bG N(i)=N(i')$. It gives a functor $\bG:\bA\md\to\bB\md$. Obviously $\bG\bF M\simeq M$,
  hence $\bF M\simeq \bF M'$ implies that $M\simeq M'$.
  
  Let now $\bA$ be obtained from $\bB$ by gluing vertices $i$ and $j$. As above, we suppose that
  there are no loops at these vertices. For a $\bB$-module $M$ set $\bF M(k)=M(k)$ if $k\ne(ij)$, 
  $\bF M(ij)=M(i)\+M(j)$, $\bF M(\al)=M(\al))$ if $\al\notin\Ar(i)\cup\Ar(j)$, 
  $\bF M(\al)=\mtr{M(\al)&0}$ \Big(or $\mtr{0&M(\al)}$\Big) if $\al^-=i$ 
  (respectively $\al^-=j$) and
  $\bF M(\be)=\mtr{M(\be)\\0}$ \Big(or $M(\be)=\mtr{0\\M(\be)}$\Big) if $\be^+=i$ (respectively 
  $\be^+=j$). If $f:M\to M'$, set $\bF f(k)=f(k)$ if $k\ne(ij)$ and $f(ij)=f(i)\+f(j)$. 
  It gives an exact functor $\bF:\bB\md\to\bA\md$. Suppose that $\phi:\bF M\ito \bF M'$,
 \begin{align*}
 \phi(ij)&=\mtr{\phi_{11}&\phi_{12}\\\phi_{21}&\phi_{22}},\\
 \phi^{-1}(ij)&=\mtr{\psi_{11}&\psi_{12}\\\psi_{21}&\psi_{22}}.
 \end{align*} 
 Then
 \begin{align*}
 \phi_{11}M(\be)&=M'(\be)\phi(k) \text{ and } \phi_{21}M(\be)=0 \text{ if } \be:k\to i,\\
 \phi_{22}M(\be)&=M'(\be)\phi(k) \text{ and } \phi_{12}M(\be)=0 \text{ if } \be:k\to j,\\
 M'(\al)\phi_{11}&=\phi(k)M(\al) \text{ and } M'(\al)\phi_{12}=0 \text{ if } \al:i\to k,\\
 M'(\al)\phi_{22}&=\phi(k)M(\al) \text{ and } M'(\al)\phi_{21}=0 \text{ if } \al:j\to k.
 \end{align*}
 and analogous relations hold for the components of $\phi^{-1}(ij)$ with interchange of $M$
 and $M'$. We suppose that $M$ has no direct summands $\oB_i$ and $\oB_j$.
 It immediately implies that $\bap_{\al^-=i}\Ker M(\al)\sbe\sum_{\be^+=i}\im M(\be)$ and
 $\bap_{\al^-=j}\Ker M(\al)\sbe\sum_{\be^+=j}\im M(\be)$. If $M'$ also contains no direct 
 summands $\oB_i$ and $\oB_j$, it satisfies the same conditions. Therefore
 \[
 \textstyle
 \im\psi_{21}\sbe\bap_{\al^-=j}\Ker M(\al)\sbe \sum_{\be^+=j}\im M(\be),
 \]
 whence $\phi_{12}\psi_{21}=0$ and $\phi_{11}\psi_{11}=1$. Quite analogously, 
 $\phi_{22}\psi_{22}=1$ and the same holds if we interchange $\phi$ and $\psi$. Therefore
 we obtain an isomorphism $\tilde\phi:M\ito M'$ setting $\tilde\phi(i)=\phi_{11}$,
 $\tilde\phi(j)=\phi_{22}$ and $\tilde\phi(k)=\phi(k)$ if $k\notin\{i,j\}$.
 \end{proof}
 
 \begin{corol}\label{improve} 
 If an algebra $\bA$ is obtained from $\bB$ by gluing or blowing up and $\bB$ is representation 
 infinite or wild, then so is also $\bA$.
 \end{corol}
 
\begin{proof}[Proof of Theorem \ref{main}]
  We have already proved the ``if'' part of the theorem. So we only have to show that all other nodal
  algebras are wild. Moreover, we can suppose that there were no inessential gluings during the 
  construction of a nodal algebra $\bA$. As $\bA$ is neither gentle nor quasi-gentle, there must 
  be at least one exceptional gluing. Hence $\bA$ is obtained from an algebra $\bB$ of the form 
  \eqref{case1} or \eqref{case2} by some additional gluings (not inessential) or blowings up. 
  One easily sees that any blowing up of $\bB$ gives a wild algebra. Indeed, the crucial 
  case is when $n=1$, $m=l=0$ and we blow up the end of the arrow $\be$. Then, after reducing 
  $\al_1$ and $\ga$, just as in the proof of Theorem~\ref{except}, we obtain for the non-zero 
  parts of the two arrows obtained from $\be$ the problem of the pair of posets $(1,1)$ and 
  $\dS_1$ (see page \pageref{ds}), which is wild by \cite[Theorem 2.3]{kl1}. The other cases are 
  even easier. 
  
  Thus no blowing up has been used. Suppose that we glue the ends of $\be$ (or some $\be_k$)
  and $\ga$ (or some $\ga_k$). Then, even for $n=1,\,m=l=0$, we obtain the algebra
  \[
   \xymatrix{ \ccdot \ar@(dl,ul)[]^{\al} \ar@/^/[rr]^\be \ar@/_/[rr]_\ga && \ccdot }\quad
    \al^2=\be\al=0
  \]
  (or its dual). Reducing $\al$, we obtain two matrices of the forms
  \begin{align*}
    \be=\mtr{0&B_2&B_3} \text{ and } \ga=\mtr{G_1&G_2&G_3}.
  \end{align*}
  Given another pair $(\be',\ga')$ of the same kind, its defines an isomorphic representation
  \iff there are invertible matrices $X$ and $Y$ such that $X\be=\be'Y$ and $X\ga=\ga'Y$,
  and $T$ is of the form
  \[
   Y=\mtr{Y_1&Y_3&Y_4\\0&Y_2&Y_5\\0&0&Y_1},  
  \]
  where the subdivision of $Y$ corresponds to that of $\be,\ga$. The Tits form of this matrix problem
  (see \cite{ds}) is $Q=x^2+2y_1^2+y_2^2+2y_1y_2-3xy_1-2xy_2$. As $Q(2,1,1)=-1$, this matrix problem
  is wild. Hence the algebra $\bA$ is also wild. Analogously, one can see that if we glue ends of some
  of $\be_i$ or $\ga_i$, we get a wild algebra (whenever this gluing is not inessential). Gluing of an 
  end of some $\al_i$ with an end of $\be$ or $\ga$ gives a wild quiver algebra as a subalgebra
  (again if it is not inessential). Just the same is in the case when we glue ends of some $\al_i$
  so that this gluing is not inessential and $n>3$. It accomplishes the proof.
 \end{proof}

\end{document}